\documentclass[11pt]{amsart}

\usepackage{enumitem}

\usepackage{hyperref}

\usepackage{bookmark}

\usepackage{newtxtext, kotex}

\usepackage[
textwidth=.85in,
textsize=footnotesize,
colorinlistoftodos]
{todonotes} 

\usepackage{verbatim} 

\usepackage[all,cmtip]{xy} 
\usepackage{multicol}

\usepackage{amsmath,amsfonts,amssymb,mathrsfs}
\usepackage{amsthm,thmtools}

\numberwithin{equation}{section}

\theoremstyle{definition}
\newtheorem{definition}[equation]{Definition}
\newtheorem{note}[definition]{Note}
 
\theoremstyle{theorem}
\newtheorem{theorem}[definition]{Theorem}
\newtheorem{lemma}[theorem]{Lemma}
\newtheorem{proposition}[theorem]{Proposition}
\newtheorem{corollary}[theorem]{Corollary}

\newcommand{\bE}{\mathbb{E}}
\newcommand{\bF}{\mathbb{F}}
\newcommand{\bG}{\mathbb{G}}
\newcommand{\bH}{\mathbb{H}}

\newcommand{\bL}{\mathbb{L}}

\newcommand{\bQ}{\mathbb{Q}}
\newcommand{\bR}{\mathbb{R}}
\newcommand{\bS}{\mathbb{S}}

\newcommand{\bZ}{\mathbb{Z}}

\newcommand{\calA}{{\mathcal{A}}}

\newcommand{\calE}{{\mathcal{E}}}
\newcommand{\calF}{{\mathcal{F}}}

\newcommand{\calK}{{\mathcal{K}}}

\newcommand{\calO}{{\mathcal{O}}}

\newcommand{\calS}{{\mathcal{S}}}

\newcommand{\calX}{{\mathcal{X}}}

\DeclareMathAlphabet{\eus}{U}{eus}{m}{n}
\let\catsymbfont\eus

\newcommand{\aA}{{\catsymbfont{A}}}

\newcommand{\aF}{{\catsymbfont{F}}}

\newcommand{\aS}{{\catsymbfont{S}}}

\newcommand{\al}{\alpha}

\newcommand{\epz}{\varepsilon}

\newcommand{\phz}{\varphi}

\newcommand{\ka}{\kappa}

\newcommand{\ta}{\tau}

\newcommand{\GA}{\Gamma}

\newcommand{\SI}{\Sigma}

\newcommand{\Ab}{\ensuremath{ \aA b }}

\newcommand{\Sp}{\ensuremath{ \aS {\rm p} }}

\newcommand{\Fin}{\ensuremath{ \aF in }}

\newcommand{\Sh}{{\rm Sh}}
\newcommand{\Hyp}{{\rm Hyp}}
\newcommand{\spec}{{\rm spec\ }}
\newcommand{\et}{\textup{\'et}}

\newcommand{\ox}{\otimes}   
\newcommand{\op}{\oplus}    
\newcommand{\w}{\wedge}    

\newcommand{\ef}{E_\infty}

\let\overto\xrightarrow

\def\quickop#1{\expandafter\newcommand\csname #1\endcsname{\operatorname{#1}}}
\quickop{map} \quickop{Map} \quickop{Hom} \quickop{End} \quickop{Aut} \quickop{Tel} \quickop{Mic}
\quickop{Ext} \quickop{Tor} \quickop{Id} \quickop{Coker} \quickop{Ker}
\quickop{Lim} \quickop{Colim} \quickop{Holim} \quickop{Hocolim}
\quickop{id} \quickop{tel} \quickop{mic} \quickop{coker}
\quickop{holim} \quickop{hocolim} \quickop{im}
\quickop{Ho} 
\quickop{Fib} \quickop{Cof}

\newcommand{\colim}{\operatorname*{colim}}

\title{K-theoretic Tate-Poitou duality at prime 2}
\author{Myungsin Cho}
\address{Department of Mathematics, Indiana University, Bloomington, IN \ 47405}
\email{myuncho@iu.edu}
\date{}

\begin{document}
\maketitle

\begin{abstract}
We extend the result of Blumberg and Mandell on K-theoretic Tate-Poitou duality at odd primes which serves as a spectral refinement of the classical arithmetic Tate-Poitou duality.
The duality is formulated for the $K(1)$-localized algebraic K-theory of the ring of $p$-integers in a number field and its completion using the $\mathbb{Z}_p$-Anderson duality.
This paper completes the picture by addressing the prime 2, where the real embeddings of number fields introduce extra complexities.
As an application, we identify the homotopy type at prime 2 of the homotopy fiber of the cyclotomic trace for the sphere spectrum in terms of the algebraic K-theory of the integers.
\end{abstract}

\section{Introduction}
Algebraic K-theory is an invariant of a ring that appears in many branches of mathematics, such as number theory, algebraic geometry, geometric topology and differential topology.
For instance, Kurihara \cite{MR1145807} showed that the vanishing of $K_{4n}(\bZ)$ is equivalent to proving the Kummer-Vandiver conjecture in number theory.
On the other hand, on the topological side, the algebraic K-theory of the sphere spectrum $\bS$  is a fundamental object to study in both algebraic and differential topology because its direct summand, the smooth Whitehead spectrum $Wh^\text{diff}(\ast)$, describes differential topology of high dimensional highly connected manifolds \cite{MR3202834}.

Let $F$ be a number field and $p$ be a prime integer.
Then each prime $\nu$ in the ring of integers $\calO_F$ lying over $p$ defines a valuation and so a completion map $\calO_F[\tfrac{1}{p}] \to F_\nu$.
Let $\ka$ be the induced map on $K(1)$-local algebraic K-theory
 \[\ka \colon L_{K(1)}K(\calO_F[\tfrac{1}{p}]) \to \prod_{\nu | p} L_{K(1)}K(F_\nu),\]
 where $K(1)$ is the first Morava K-theory.
When $p$ is an odd prime, the homotopy type of the homotopy fiber $\Fib(\ka)$ was determined by the work of Blumberg and Mandell which gives a spectral refinement of the arithmetic Tate-Poitou duality \cite[I.4.10]{MR2261462}.
\begin{theorem}[\cite{MR4121155}]\label{theoremodd}
 If $p$ is odd, there is a canonical weak equivalence between the homotopy fiber $\Fib(\ka)$ and $\SI^{-1} I_{\bZ_p} L_{K(1)}K(\calO_F[\tfrac{1}{p}])$, the desuspension of the $\bZ_p$-Anderson dual spectrum of the $K(1)$-local algebraic K-theory of $\calO_F[\tfrac{1}{p}]$.
 The weak equivalence is adjoint to the map
 \[ \Fib(\ka) \w L_{K(1)}K(\calO_F[\tfrac{1}{p}]) \to \SI^{-1}I_{\bZ_p}\bS\]
 induced by the $L_{K(1)}K(\calO_F[\tfrac{1}{2}])$-action on $\Fib(\ka)$ followed by a canonical map $u_{\calO_F}\colon \Fib(\ka)  \to \SI^{-1} I_{\bZ_p}\bS$.
\end{theorem}
Related work in this direction appears in Clausen’s thesis \cite{MR3211458}, where he constructs a similar duality map connected to Gross-Hopkins duality. See also recent work of Braulig \cite{braunling}, which investigates analogous phenomena in the algebraic K-theory of locally compact modules.

The $\bZ_p$-Anderson dual of the sphere spectrum $I_{\bZ_p}\bS$ is defined as the $p$-completion of the desuspension of the Brown-Comenetz dual of the sphere spectrum
\[I_{\bZ_p}\bS =(\SI^{-1} I_{\bQ/\bZ}\bS)_p^\w\]
where the Brown-Comenetz dual spectrum $I_{\bQ/\bZ}\bS$ represents the Pontryagin dual of $\pi_0$.
Specifically, for any spectrum $X$, it satisfies
\[ [X, I_{\bQ/\bZ}\bS] \cong \Hom(\pi_{0} X, \bQ/\bZ) = (\pi_0 X)^\vee\]
 where $[-,-]$ denotes maps in the stable homotopy category and $(-)^\vee$ denotes the Pontryagin dual.
These, indeed, are homotopical refinements of the Pontryagin duality of finite abelian groups and satisfies analogous idempotent duality isomorphisms in appropriate subcategories \cite{MR405403,MR388375}.
The \autoref{theoremodd} gives rise to an isomorphism
\begin{align*}
 \pi_n \Fib(\ka)  & \cong \pi_{n} (\SI^{-1} I_{\bQ/\bZ} (L_{K(1)} K(\calO_F[\tfrac{1}{p}]) \w M_{\bZ/p^\infty}))   \\
 		& \cong (\pi_{-1-n} (L_{K(1)} K(\calO_F[\tfrac{1}{p}]) \w M_{\bZ/p^\infty}) )^\vee,
\end{align*}
where $M_{\bZ/p^\infty}$ denotes the Moore spectrum for $\bZ/p^\infty$.
Hence, the long exact sequence of homotopy groups associated to the fiber sequence of the map $\ka$ recovers the 9-term Tate-Poitou long exact sequence.

When $p=2$, the methods of \cite{MR4121155} apply only when $F$ is totally imaginary.
The purpose of this paper is to extend this result to prime 2 without this assumption on $F$.

\begin{theorem}[Two-primary K-theoretic Tate-Poitou duality]\label{mainthm}
Let $F$ be a number field, $\calO_F$ its ring of integers and $\calS$ the set of finite primes of $\calO_F$ above 2.
There is a canonical weak equivalence between the homotopy fiber $\Fib(\ka)$ of the completion map in $K(1)$-local algebraic K-theory
 \[\ka\colon L_{K(1)}K(\calO_F[\tfrac{1}{2}]) \to \prod_{\nu \in \calS} L_{K(1)}K(F_\nu)\]
and $\SI^{-1} I_{\bZ_2} L_{K(1)}K(\calO_F[\tfrac{1}{2}])$, the desuspension of the $\bZ_2$-Anderson dual spectrum of the $K(1)$-local algebraic K-theory of $\calO_F[\tfrac{1}{2}]$.
 The weak equivalence is adjoint to the map
 \[ \Fib(\ka) \w L_{K(1)}K(\calO_F[\tfrac{1}{2}]) \to \SI^{-1}I_{\bZ_2}\bS\]
 induced by the $L_{K(1)}K(\calO_F[\tfrac{1}{2}])$-action on $\Fib(\ka)$ followed by a canonical map $u_{\calO_F}\colon \Fib(\ka)  \to \SI^{-1} I_{\bZ_2}\bS$, constructed as (\ref{dualityclass}).
\end{theorem}

The main obstruction to extending the Blumberg-Mandell's work to the $p=2$ case is the presence of real embeddings.
In the case of odd primes, for $R = \calO_F[\tfrac{1}{p}]$ or $F_\nu$, the Thomason spectral sequence \cite{Thomason}
\begin{equation}\label{thomasonss}
 E_2^{s,t} =  H^s_\et(R;\bZ/p^n(-\tfrac{t}{2})) \Longrightarrow \pi_{t-s}( L_{K(1)} K(R) \w M_{\bZ/p^n})
\end{equation} 
collapses at $E_2$-page for degree reasons with trivial extensions.
This gives a canonical isomorphism between the homotopy groups of the $K(1)$-localization of algebraic K-theory spectra and Jannsen’s continuous \'etale cohomology groups:
\begin{equation}\label{thomasonssodd}
 \pi_n L_{K(1)}K(R) \cong 
\begin{cases}
 H^0_\et(R; \bZ_p(\tfrac{n}{2})) \op H^2_\et(R; \bZ_p(\tfrac{n}{2}+1)) & n \text{ even} \\
 H^1_\et(R; \bZ_p(\tfrac{n+1}{2}))  & n \text{ odd}
\end{cases}
\end{equation}
With the above identification, Albert-Brauer-Hasse-Noether sequence for $\calO_F[\tfrac{1}{p}]$ then provides the construction of the canonical map $u_{\calO_F}$ in the \autoref{theoremodd}; see \cite[1.7]{MR4121155}.
The other essential algebraic input of Blumberg-Mandell is the specific form of Artin-Verdier duality at odd primes 
\begin{equation}
\begin{split}
H^p_\et(\calO_F; j_! \bZ/p^n(t)) \ox H^{3-p}_\et(\calO_F[\tfrac{1}{p}], \bZ/p^n(1-t)) \qquad \\
 \to H^3_\et(\calO_F; \bG_{m}) \cong \bQ/\bZ
\end{split}
\end{equation}
 which arises from the isomorphism between the \'etale cohomology $H^\ast_\et(\calO_F;j_! \bZ/p^n(t))$ and \'etale cohomology groups with compact support $H^\ast_c(\calO_F[\tfrac{1}{p}];\bZ/p^n(t))$ where $j\colon \spec \calO_F[\tfrac{1}{p}] \to \spec \calO_F$ is the canonical open embedding.
In the 2-primary case, when $F$ admits a real embedding, both observations fail.
The mod 2 \'etale cohomological dimension of $\calO_F$ is infinite and the spectral sequence (\ref{thomasonss}) no longer collapses at the $E_2$-page, but it collapses at the $E_4$-page; see \cite[2.10]{MR1803955}.
Therefore the homotopy groups $\pi_\ast L_{K(1)}K(\calO_F)$ no longer admit a simple \'etale cohomological description as in the odd prime case (\ref{thomasonssodd}).
Instead, the spectral sequences provides a description in terms of subquotients, subject to extension problems; see \cite[0.6]{MR1697095} and \cite[2.17]{MR1803955}.
Furthermore, the \'etale cohomology group $H^\ast_\et(\calO_F;j_!M)$ differs from the compactly supported cohomology $H^\ast_c(\calO_F[\tfrac{1}{p}];M)$ by  a direct sum $\bigoplus_{\nu} H^n_{T}(\bR;\nu_\ast\calA)$ through the long exact sequence:
\begin{equation} \label{les}
 \cdots  \to H^\ast_c(\calO_F[\tfrac{1}{p}];M) \to H^\ast_\et(\calO_F;j_!M) \to  \bigoplus_{\nu} H^n_{T}(\bR;\nu_\ast\calA) \to \cdots 
\end{equation}
where the later term is Tate cohomology group over real embeddings  $\nu\colon \calO_F \to \bR$ which vanishes in the case of odd primes \cite[II.2.3]{MR2261462}.
Both of these phenomena are closely tied to the presence of real embeddings, and their étale cohomological behavior is characterized by Tate’s theorem.
\begin{theorem}[{\cite[3.1(c)]{MR0175892}}]\label{tate}
For a 2-torsion sheaf $\calA$ on $(\spec \calO_F)$
the map 
\[ H^n_{\et}(\calO_F[\tfrac{1}{2}];\calA ) \to \bigoplus_{\nu} H^n_{\et}(\bR;\nu_\ast\calA)\]
induced by the real embedding $\nu\colon \calO_F \to \bR$, is an isomorphism on $n\ge 3$.
\end{theorem}

As in \cite{MR4121155}, the immediate corollary of the \autoref{mainthm} for the $F =\bQ$ case is the identification of the homotopy fiber of the cyclotomic trace of the sphere spectrum.
\begin{corollary}\label{maincoro}
  The connective cover of the homotopy fiber of the cyclotomic trace $\ta_\bS\colon K(\bS)_2^\w \to TC(\bS)_2^\w$ is canonically weakly equivalent to  $\SI^{-1} I_{\bZ_2}\left(L_{K(1)} K(\bZ)\right)$, the connective cover of the desuspension of the $\bZ_p$-Anderson dual of the $K(1)$-localized algebraic K-theory of the integers.
  
\end{corollary}
\noindent This result follows from the K-theoretic reduction of the homotopy fiber $\Fib(\ta_\bS)$ to $\Fib(\ka_\bZ)$, building on the works of Dundas \cite{MR1607556} and Hesselholt-Madsen \cite{MR1410465} together with the (affirmed) Quillen-Lichtenbaum conjecture.
For a more detailed discussion, we refer the reader to the introduction of \cite{MR4121155}.

In the case of $F =\bQ$, the \autoref{mainthm} was proved by Blumberg-Mandell-Yuan \cite[5.3]{BMY} without identifying the weak equivalence.
This could be achieved by first resolving the spectra $K(\bZ)$ and $\Fib(\ka)$ as fiber sequences involving the connective topological K-theory spectra $ko$ and $ku$ using Rognes' result \cite{MR1923990}, and then analyzing these fiber sequences with the self Anderson duality of $KO$ and $KU$.
They left the complete statement as a conjecture \cite[5.4]{BMY} and identify the homotopy fiber of chromatic completion map of the K-theory of $p$-local sphere $\Fib(K(\bS_{(p)})_{(p)} \to \lim L_n K(\bS_{(p)}))$.
Our main result affirms this conjecture and proves the following corollary.
\begin{corollary}
The homotopy fiber $\Fib(K(\bS_{(2)})_{(2)} \to \lim_n L_n K(\bS_{(2)}))$ fits in a fiber sequence
\[ \Fib \to \bigvee_{\ell \ne 2} \SI^{-2} I_{\bZ/2^\infty} K(\bF_\ell) \to \SI^{-2} I_{\bZ/2^\infty} K(\bZ)\]
where the wedge is over primes $l\ne 2$, and the map $I_{\bZ/2^\infty} K(\bF_\ell) \to I_{\bZ/2^\infty} K(\bZ)$ is Brown-Comenetz dual to the map $K(\bZ) \to K(\bF_\ell)$ induced by $\bZ \to \bF_\ell$.
\end{corollary}

In this paper, the author resolves the difficulties arising at the prime 2 using the technique of the completed \'etale site introduced by Thomas Zink \cite{zink}.
Intuitively, this technique formally adjoins real points to the \'etale site $Y:=(\spec \calO_F)_\et$ to form a new Grothendieck site $\overline{Y_\et}$, called a completed \'etale site.
In this construction, one may think of $Y_\et$ as an open subset of $\overline{Y_\et}$ whose closed complement is given by the real embeddings $Y_\infty$:
\begin{equation}\label{openclosed}
 \xymatrix{
Y_\et \ar[r]^-{j}& \overline{Y}_\et & Y_\infty \ar[l]_-{i}
}\end{equation} 
This diagram induces recollement on the categories of sheaves of abelian groups (\ref{zinkrecol}).
Notably, in $\overline{Y_\et}$, the real points behave as if they were geometric points, and this site has cohomological dimension at most 3 for 2-torsion sheaves; this fact is equivalent to Tate's \autoref{tate}.
Within this formalism, Zink proved the Artin-Verdier-Zink duality theorem. 

This paper is organized as follows.
In \autoref{s1}, we review the K-theoretic Tate local duality theorem proved in \cite[1.4]{MR4121155} and provide a construction of the canonical map $u_{\calO_F}\colon \Fib(\ka) \to \SI^{-1} I_{\bZ_2}$.
In \autoref{s2}, we introduce Zink's construction on completed \'etale site and analyze hypersheaves of spectra.
We then give a hypercohomological description of the homotopy fiber $\Fib(\ka)$.
In \autoref{s3}, we construct spectral sequences associated to $\Fib(\ka)$ and $L_{K(1)}K(\calO_F[\tfrac{1}{2}])$, and prove the \hyperref[mainthm]{main theorem}.

\subsection*{Acknowledgements}
I am deeply grateful to my advisor, Michael Mandell, for suggesting this project and providing invaluable guidance.
I would also like to thank Taeyeup Kang, Sangmin Ko, Noah Riggenbach, and Ningchuan Zhang for their helpful discussions.
In particular, I am especially grateful to Ningchuan Zhang for his detailed comments and suggestions on revisions.
I also appreciate Oliver Brauling for his thoughtful feedback.
I thank the anonymous referee for helpful comments that improved the clarity of the paper.
This work was partially supported by a College of Arts + Sciences Dissertation Research Fellowship from Indiana University.
Additionally, support was provided in part by NSF grants DMS-1811820, DMS-2052846, and DMS-2104348.

\section{Local duality and construction of the duality class}\label{s1}
In this section, we review the K-theoretic Tate local duality theorem stated in \cite[\S1]{MR4121155}.
We also include a more detailed explanation of how the Thomason spectral sequence together with arithmetic Tate local duality yields an isomorphism between homotopy groups of $K(1)$-local algebraic K-theory of a local field, which suggests spectrum level duality.
After that, we construct the canonical map $u_{\calO_F}$ that appears in \autoref{mainthm} by analyzing the Thomason spectral sequence together with the Albert-Brauer-Hasse-Noether exact sequence.

\subsection{K-theoretic Tate local duality}
Let $k$ be a field of fractions of a complete discrete valuation ring whose residue field is finite of characteristic $p$.
The arithmetic Tate local duality states that the cup product in \'etale cohomology gives an isomorphism:
\begin{equation}\label{tatelocalduality}
 H^i_\et(k;M) \cong  (H^{2-i}_\et(k; M^\vee (1)))^\vee
\end{equation}
where $M$ is a finite Galois module and $(-)^\vee = \Hom(-,\bQ/\bZ)$ denotes the Pontryagin dual. 
This isomorphism together with Thomason's identification gives isomorphisms between homotopy groups of $K(1)$-local algebraic K-theory of the local field $k$ (possibly with coefficients), which can be lifted to a weak equivalence at the spectrum level.
More precisely, when $M = \bZ/p^m(n)$ and $m\to \infty$, the duality isomorphisms (\ref{tatelocalduality}) together with the identification (\ref{thomasonssodd}) gives abstract isomorphisms of homotopy groups: 
\begin{equation}\label{localiso}
\begin{split}
\pi_{n} L_{K(1)}K(k)& \cong \pi_{-n} (L_{K(1)}K(k)  \w M_{\bZ/p^\infty} )^\vee \\
&\cong \pi_{n} I_{\bQ/\bZ}(L_{K(1)}K(k) \w M_{\bZ/p^\infty}) \\
&\cong \pi_n I_{\bZ_p} (L_{K(1)}K(k)),
\end{split}
\end{equation}
where $M_{\bZ/p^\infty}$ denotes the Moore spectrum for $\bZ/p^\infty$.
Blumberg and Mandell lifted this isomorphism of homotopy groups to a weak equivalence between spectra and stated the K-theoretic local duality theorem:
\begin{theorem}[{\cite[1.4]{MR4121155}}]
The isomorphism of homotopy groups (\ref{localiso}) is induced from a canonical weak equivalence
\[L_{K(1)}K(k) \simeq I_{\bQ/\bZ}(L_{K(1)}K(k) \w M_{\bZ/p^\infty})\]
whose adjoint is the composite
 \[L_{K(1)}K(k)  \w L_{K(1)}K(k) \w M_{\bZ/p^\infty} \to L_{K(1)}K(k) \w M_{\bZ/p^\infty} \to I_{\bQ/\bZ}\bS\]
where the former map is the $\ef$ multiplication of K-theory and the latter map is a canonical map.
\end{theorem}
The latter map is induced by the short exact sequence
\begin{equation}\label{thomasonssloc}
 0 \to H^2_\et(k; \bZ/p^\infty(1)) \to  \pi_0( L_{K(1)}K(k)\w M_{\bZ/p^\infty})  \to H^0_\et(k; \bZ/p^\infty)\to 0
\end{equation}
arising from Thomason's spectral sequence, whose canonical splitting is given by the section $H^0_\et(k; \bZ/p^\infty)  \to \pi_0( L_{K(1)}K(k)\w M_{\bZ/p^\infty})$ coming from the unit map $\bS \to K(k)$.
The canonical retraction
\begin{equation}\label{localdualitymap}
 \pi_0( L_{K(1)}K(k)\w M_{\bZ/p^\infty})  \to H^2_\et(k; \bZ/p^\infty(1))  \to H^2_\et(k; \bQ/\bZ(1)) \cong \bQ/\bZ
\end{equation}
followed by the Hasse invariant isomorphism\cite[XIII\S3,p. 193]{MR0554237}, defines the desired map of spectra.

\subsection{Construction of the duality class map}
In order to construct the global duality map $u_{\calO_F}$, we use the Albert-Brauer-Hasse-Noether exact sequence instead of the Hasse invariant.
Let $F$ be a number field, $\calO_F$ be its ring of integers and let $r_1$ be the number of distinct real embeddings.
We denote $\calS$ by the set of prime ideals in $\calO_F$ dividing 2, and let $\calS_\infty$ be the union of $\calS$ and the set of Archimedean places of $F$.
The 2-torsion version of the Albert-Brauer-Hasse-Noether exact sequence for $\calO_F[\tfrac{1}{2}]$ then takes the form
\[0  \to H^2_\et(\calO_F[\tfrac{1}{2}];\bQ_2/\bZ_2(1))  \to  \bigoplus\limits_{\nu\in \calS_\infty}H^2_\et(F_\nu;\bQ_2/\bZ_2(1))  \overto{h} \bQ_2/\bZ_2 \to 0.\]
The direct summand $H^2_\et(F_\nu;\bQ_2/\bZ_2(1))$ is isomorphic to $\bQ_2/\bZ_2$ for finite primes (by the Hasse invariant isomorphism), $\bZ/2$ for the real places, and 0 for the complex places.
The latter map $h$ is given by the sum of the injections.

The Thomason descent spectral sequence (\ref{thomasonss}) associated with the ring $\calO_F[\tfrac{1}{2}]$ takes the form
\[E_2^{s,-t} = H^s_\et(\calO_F[\tfrac{1}{2}];\bQ_2/\bZ_2(\tfrac{t}{2})) \Rightarrow \pi_{t-s} L_{K(1)} K(\calO_F[\tfrac{1}{2}] \w M_{\bQ_2/\bZ_2})\]
and this collapses at the $E_4$-page \cite[2.10]{MR1803955}.
In particular, the map of the spectral sequence associated with the real embeddings $\calO_F[\tfrac{1}{2}]\to \bR$ induces a commutative diagram on the $E_3$-page
\[\xymatrix{
 H^2_\et(\calO_F[\tfrac{1}{2}];\bQ_2/\bZ_2(1)) \ar[r]  \ar[d]^-{d_3} & \bigoplus^{r_1} H^2_\et(\bR;\bQ_2/\bZ_2(1))  \ar[d]^-{d_3}_-\cong \\
 H^5_\et(\calO_F[\tfrac{1}{2}];\bQ_2/\bZ_2(2))  \ar[r]^-\cong & \bigoplus^{r_1} H^5_\et(\bR;\bQ_2/\bZ_2(2)) 
}\]
where the horizontal maps are maps of the spectral sequence on $E_3$-page and the vertical maps are the differential $d_3\colon E^{2,-2}_3 \to E^{5,-4}_3$.
The lower horizontal map is the isomorphism in \autoref{tate} and the right vertical isomorphism follows from the descent spectral sequence of the real numbers $\bR$, which is the homotopy fixed point spectral sequence associated with $KU^{hC_2}$ with complex conjugate action.
This identifies $E_4^{2,-2}$ with the summand of the 2-primary Brauer group corresponding to the finite primes above 2
\[H^2_\et(\calO_F[\tfrac{1}{2}];\bQ_2/\bZ_2(1))  \cong (\bQ_2/\bZ_2)^{s-1}\] where $s$ is the  number of places lying above 2.
Hence, we obtain a map of short exact sequences
\begin{equation}\label{ses1}
 \xymatrix@R=1em@C=1.3em{
0 \ar[r] & E_4^{2,-2} \ar[r]  \ar[d] &  \prod\limits_{\nu\in \calS}H^2_\et(F_\nu;\bQ_2/\bZ_2(1))   \ar[d] \ar[r] & \bQ_2/\bZ_2 \ar@{=}[d] \ar[r] & 0\\
0 \ar[r] & H^2_\et(\calO_F[\tfrac{1}{2}];\bQ_2/\bZ_2(1)) \ar[r] &  \prod\limits_{\nu\in \calS_\infty}H^2_\et(F_\nu;\bQ_2/\bZ_2(1))  \ar[r] & \bQ_2/\bZ_2 \ar[r] & 0
}\end{equation}
where the left and the middle vertical maps are natural inclusions.
The map of the descent spectral sequence associated to the non-Archimedean completion map gives the diagram
\[\xymatrix@R=1.3em{
0 \ar[d] & 0 \ar[d]  \\
E_4^{2,-2} \ar[d] \ar[r] & \prod\limits_{\nu\in \calS}H^2_\et(F_\nu;\bQ_2/\bZ_2(1))  \ar[d] \\
\pi_{0} (L_{K(1)} K(\calO_F[\tfrac{1}{2}])\w M_{\bQ_2/\bZ_2}) \ar[d]  \ar[r] & \prod\limits_{\nu\in \calS} \pi_{0} (L_{K(1)} K(F_\nu)\w M_{\bQ_2/\bZ_2}) \ar[d] \\
H^0_\et(\calO_F[\tfrac{1}{2}];\bQ_2/\bZ_2) \ar[r]  \ar[d] &  \prod\limits_{\nu\in \calS} H^0_\et(F_\nu;\bQ_2/\bZ_2) \ar[d]\\
0 & 0
}\]
and the upper horizontal map is the inclusion map that appears in the upper row of the diagram (\ref{ses1}).
Hence, we get a natural map from $\bQ_2/\bZ_2$ to the cokernel
\[C= \coker\left(\pi_{0} (L_{K(1)} K(\calO_F[\tfrac{1}{2}]) \w M_{\bQ_2/\bZ_2}) \to \prod\limits_{\nu\in \calS} \pi_{0} (L_{K(1)} K(F_\nu) \w M_{\bQ_2/\bZ_2}) \right).\]
This map, combined with the natural map $C \to \pi_{-1} \Fib(\ta\w M_{\bQ_2/\bZ_2})$ induced from the long exact sequence of homotopy group gives a map 
\[\bQ_2/\bZ_2 \to \pi_{-1} (\Fib\ta \w M_{\bQ_2/\bZ_2}).\]
By \cite[1.5]{MR4121155}, the K-theory transfer associated to the inclusion of number fields extends to a map on $\Fib(\ka)$ and the inclusion $\bQ \subset F$ induces a canonical retraction
\[\xymatrix{
\bQ_2/\bZ_2 \ar[r] \ar@{=}[d] & \pi_{-1} (\Fib\ta_F\w M_{\bQ_2/\bZ_2}) \ar[d] \\
\bQ_2/\bZ_2 \ar[r]^-\cong & \pi_{-1} (\Fib\ta_\bQ\w M_{\bQ_2/\bZ_2}).
}\]
The composite of the retraction $ \pi_{-1} (\Fib\ta_F\w M_{\bQ_2/\bZ_2}) \to \bQ_2/\bZ_2$ with the inclusion of 2-torsions $\bQ_2/\bZ_2 \to \bQ/\bZ$ defines the desired map 
\begin{equation}\label{dualityclassadj}
 \Fib(\ka) \w M_{\bZ/2^\infty} \to \SI^{-1}I_{\bQ/\bZ}\bS
\end{equation}
or, equivalently 
\begin{equation}\label{dualityclass}
u_{\calO_F}\colon \Fib(\ka) \to \SI^{-1} I_{\bZ_2}\bS.
\end{equation}

\section{Completed small \'etale site of Zink}\label{s2}
In this section, we review Zink's construction of the completed \'etale site of the ring of integers of a number field, described in \cite{zink}.
For a more modern with detailed treatment, see \cite{conrad}.
We then use this construction to give a hypercohomological description of the homotopy fiber $\Fib(\ka)$ in \autoref{mainthm}.

\subsection{Sheaves of abelian groups on Zink's completed \'etale site}
Fix a number field $F$ and denote $Y = \spec \calO_F$.
Let $\calS$ be the set of primes lying over 2, let $U$ be the open subscheme $Y \setminus \calS = \spec \calO_F[\tfrac{1}{2}]$, and let $Z$ be the reduced closed subscheme $Y \setminus U = \coprod_{\nu|2} \spec (\calO_F / \nu)$.
The natural inclusions $j\colon U \to Y$ and $i\colon Z \to Y$ induce adjoint functors between the categories of \'etale sheaves of abelian groups:
\begin{equation}\label{abelianrecollement}
 \xymatrix@C=3em{
\Sh(U_\et;\Ab) \ar@<1.5ex>[r]^-{j_!} \ar@<-1.5ex>[r]_-{j_\ast} & \Sh(Y_\et;\Ab)  \ar[l]|-{j^\ast}  \ar@<1.5ex>[r]^-{i^\ast} \ar@<-1.5ex>[r]_-{i^!}  & \Sh(Z_\et;\Ab)  \ar[l]|-{ \ i_\ast}
}
\end{equation}
where each functor is left adjoint to the functor below it, and these form a recollement.
One consequence of the recollement is that the category of sheaves $\Sh(Y_\et;\Ab)$ is completely determined by its restrictions on $U_\et$ and $Z_\et$ with gluing data.
\begin{theorem}\cite[III.2.5]{artin}\label{recolthm}
The category $\Sh(Y_\et;\Ab)$ is equivalent to the category of triples $(M',M'',\phz)$ where $M' \in \Sh(U_\et;\Ab)$, $M'' \in \Sh(Z_\et;\Ab)$ and $\phz\colon M'' \to i^\ast j_\ast M'$.
The equivalence is given by a functor
\[M \mapsto (j^\ast M, i^\ast M, \phz_M\colon i^\ast M \to  i^\ast j_\ast j^\ast M)\]
where the gluing morphism $\phz_M$ is induced from the unit map $\epz_j\colon M \to j_\ast j^\ast M$ associated to the adjunction $j^\ast \dashv j_\ast$.
The inverse is given by the following pullback diagram:
\[\xymatrix{
M \ar[r] \ar[d] \ar@{}[dr]|(0.2)\lrcorner & i_\ast M'' \ar[d]^-{i_\ast \phz} \\
j_\ast M' \ar[r]^-{\epz_i} & i_\ast i^\ast j_\ast M'.
}\]
\end{theorem}
As a result, for a sheaf $M$ in $\Sh(Y_\et;\Ab)$, the unit map of $i^\ast \dashv i_\ast$ and the counit map of $j_! \dashv j^\ast$ fit into an exact sequence:
\begin{equation}\label{recolses}
0 \to j_! j^\ast M \to M \to i_\ast i^\ast M \to 0.
\end{equation}

In \cite{zink}, Zink introduced a Grothendieck site $\overline{Y}_\et$ by formally adjoining the real points to $Y_\et$, and presented the recollement structure as well as the Artin-Verdier duality pairing in a relevant form.
The underlying category of $\overline{Y}_\et$ has objects given by pairs $(V,T)$, where $V$ is an object in $Y_\et$ and $T\subset V_\infty = \Hom(\spec \bR, V)$ is a subset of the real points of $V$.
A morphism $(V,T) \to (V',T')$ is a morphism $V \to V'$ in $Y_\et$ that maps $T$ into $T'$.
A family $\{(V_\al,T_\al) \to (V,T)\}_\al$ is defined to be a covering if $\{V_\al \to V\}_\al$ is a covering in $Y_\et$ and $\{ T_\al \to T\}_\al$ is a covering in the category of sets, equipped with the topology generated by jointly surjective families of morphisms.
This covering then defines a Grothendieck topology, and these data form a Grothendieck site.
In particular, the projection onto the first factor $(V,T) \mapsto V$, gives rise to a functor from $\overline{Y_\et}$ to $Y_\et$ on the underlying categories. 
This in turn induces a site morphism $j_Y\colon Y_\et \to \overline{Y_\et}$.
On the other hand, each real embedding $v_k$ of $F$ defines a functor from $\overline{Y}_\et$ to $\Fin$, the category of finite sets, that assigns $(V,S) \mapsto \{s \in S: s \mapsto v_k\}$.
By considering $Y_\infty = \{v_1,v_2,\ldots,v_{r_1}\}$ as the $r_1$ copies of $\Fin$, the above construction induces a site morphism $i_Y\colon Y_\infty \to \overline{Y_\et}$.
These two site morphisms yield a diagram
\begin{equation}\label{modifiedrecollement}
 \xymatrix{
Y_\et \ar[r]^-{j_Y}& \overline{Y}_\et & Y_\infty \ar[l]_-{i_Y}
}\end{equation}
and they can be regarded as embeddings of an open subset and its closed complement, because they give rise to a recollement on the categories of sheaves \cite[1.3.2]{zink}:
\begin{equation}\label{zinkrecol}
 \xymatrix@C=4em{
\Sh(Y_\et;\Ab) \ar@<1.5ex>[r]^-{ j_{Y !} } \ar@<-1.5ex>[r]_-{ j_{Y \ast}} & \Sh(\overline{Y}_\et;\Ab)  \ar[l]|-{j_Y^\ast}  \ar@<1.5ex>[r]^-{i_Y^\ast} \ar@<-1.5ex>[r]_-{i_Y^!}  & \Sh(Y_\infty;\Ab)  \ar[l]|-{  i_{Y \ast}}.
}
\end{equation}
Therefore, the sheaf topos $\Sh(\overline{Y_\et};\Ab)$ is equivalent to the category of triples as in \autoref{recolthm} and the site $\overline{Y}_\et$ may be considered as a compactification of $Y_\et$ along the real embeddings $Y_\infty$.

A point on a sheaf topos $\Sh(Y_\et;\Ab)$ determines a point on $\Sh(\overline{Y_\et};\Ab)$ using the site morphism $j_Y$, and we call this an ordinary point.
Therefore, a stalk of a sheaf $\overline{M}$ over an ordinary point is simply a stalk of a sheaf $j_{Y}^\ast\overline{M}$ in $Y_\et$.
On the other hand, a site morphism $i_{v_k}\colon \Fin \to \overline{Y_\et}$ associated to the real embedding $v_k$ gives rise to a point $(i_{v_k}^\ast, i_{v_k \ast})$ on the topos $\Sh(\overline{Y}_\et;\Ab)$.
Hence, we refer to $v_k$ as a real point.
Given a sheaf $\overline{M}$ on $\overline{Y_\et}$, a stalk over a real point $v_k$ is computed by a colimit
\[  i_{v_k}^\ast \overline{M} = \colim_\al \overline{M}(R_\al) \cong \overline{M}( \colim_\al R_\al)= \overline{M}( \bR_{alg}) \]
where $\al$ ranges over all rings of integers $R_\al$ with $v_k(\calO_F) \subset R_\al \subset \bR$ and $\bR_{alg} = \bR \cap \overline{F}$ is the real algebraic numbers for some fixed separable closure $\overline{F}$.
By letting the inertia group $I_{v_k}$ be the cyclic group $C_2$ of order 2, we may write this stalk as the $I_{v_k}$-invariants of the stalk $\overline{M}_{\eta}$ over a generic point $\eta \in U$ so that the real points are treated as geometric points; see \cite[2.3.2]{conrad}.
Note that the union of the ordinary points and the real points forms a sufficient set of points for $\Sh(\overline{Y_\et};\Ab)$.

As in \autoref{recolthm}, a triple 
\[\left( M, M_\infty, \phz_{M}\colon M_\infty \to  i_Y^\ast j_{Y \ast} M \right)\]
 determines a sheaf in $\overline{Y_\et}$ where $M\in \Sh(Y_\et;\Ab)$ and $M_\infty \in \Sh(Y_\infty;\Ab)$.
Here, the sheaf $M_\infty$ can be written as a direct sum of abelian groups $\oplus_{k=1}^{r_1} M_{v_k}$ over each real point $v_k$.
Using the observation of the previous paragraph, the sheaf $i_Y^\ast j_{Y \ast} M$ can be identified with the direct sum
\begin{equation}\label{fixedpointidentification}
i_Y^\ast j_{Y \ast} M \cong \bigoplus_{k=1}^{r_1} (M_{\eta})^{I_{v_k}}
\end{equation}
and a sheaf on $\overline{Y}_\et$ is uniquely determined by the pullback square
\[\xymatrix{ 
\overline{M} \ar[r] \ar[d] \ar@{}[dr]|(0.2)\lrcorner& i_{Y \ast} \bigoplus_{k=1}^{r_1} M_{v_k} \ar[d]^-{i_{Y \ast} \op \phz_k} \\
j_{Y \ast} M \ar[r] & i_{Y \ast} \bigoplus_{k=1}^{r_1} (M_\eta)^{I_{v_k}}.
}\]

\subsection{Modified \'etale sheaf and its derived functor}
Under the identification of $\Sh(\overline{Y_\et};\Ab)$ with the category of triples, two functors $j_{Y \ast}$ and $j_{Y !}$ from $\Sh(Y_\et;\Ab)$ to $\Sh(\overline{Y_\et};\Ab)$ are described as follows
\begin{align*}
j_{Y \ast}M =& \left(M, \oplus_{k=1}^{r_1} (M_\eta)^{I_{v_k}} , \id \right) \\
j_{Y !}M =& (M ,0 ,0)
\end{align*}
In \cite{zink}, Zink defined the modification functor $(\widehat{-}) \colon \Sh(Y_\et;\Ab) \to \Sh(\overline{Y_\et};\Ab)$ by assigning identity functor on $Y_\et$ and $I_{v_k}$-coinvariants on each real point $v_k$.
\begin{definition}\cite[3.1.1]{zink}
 Given a sheaf $M$ on $Y_\et$, its {\it modified \'etale sheaf} $\widehat{M}$ on $\overline{Y}_\et$ is defined  by the triple
\begin{equation}\label{modifiedsheaf}
 M \mapsto \widehat{M}:= \left( M, \bigoplus_{k=1}^{r_1} (M_{\eta})_{I_{v_k}} ,  \bigoplus_{k=1}^{r_1}(M_{\eta} )_{I_{v_k}} \to \bigoplus_{k=1}^{r_1}(M_{\eta} )^{I_{v_k}} \right),
\end{equation}
where the gluing morphism is the direct sum of the norm map $Nm\colon (M_{\eta})_{I_{v_k}} \to  (M_{\eta} )^{I_{v_k}}$ for each real point $v_k$.
\end{definition}
The restriction $j_Y^\ast \widehat{M}$ onto $Y_\et$ is naturally isomorphic to $M$, and the unit map associated to the adjunction $j_{Y \ast} \dashv j_Y^\ast$ gives a natural map
\[\widehat{M} \to j_{Y \ast} j_Y^\ast \widehat{M} \overto\cong j_{Y \ast} M.\]
Since the kernel and the cokernel of this unit map are supported at the real points, we get a natural isomorphism
\[ H^s(\overline{Y}_\et; \widehat{M}) \cong H^s(\overline{Y}_\et; j_{Y \ast} M) \]
for $s \ge 1$.

Next, we define its derived functor.
The presence of the coinvariant functor in (\ref{modifiedsheaf}), which is not left exact, implies that the functor $\GA(\overline{Y}_\et; (\widehat{-}))$ fails to be left exact as well. 
Therefore, its derived functor extends to negative degrees to account for this non-exactness.
Consider a unit map $\epz_i$ associated to the (derived) adjunction $i_Y^\ast \dashv i_{Y \ast}$ 
\[R j_{Y \ast} (M) \overto{\epz_i} i_{Y \ast} i_Y^\ast R j_{Y \ast} (M)\] 
where $R j_{Y \ast}$ is the right derived functor.
The shifted mapping cone of $\epz_i$ exhibits $j_{Y !}M$ in the derived category.
The chain complex  $i_Y^\ast R j_{Y \ast} (M)$ is quasi-isomorphic to the Godement resolution of $i_Y^\ast j_{Y \ast} M$ and this is the direct sum of standard complex $C^\bullet(M_{\eta})$ which computes Galois cohomology of $M_{\eta}$ as $I_{v_k}$-module; see \cite[III.1.20]{MR0559531}.
The derived modification functor is defined as a shifted mapping cone of the unit morphism $\epz_i$ followed by the canonical map
\[\mathrm{can}\colon C^\bullet(M_{\eta}) \to \widehat{C}^\bullet(M_{\eta})\]
where $\widehat{C}^\bullet(M_{\eta})$ is the complete standard complex which computes Tate cohomology.
\begin{definition}\label{derivedmodifiedsheaf}
 Given an abelian sheaf $M$ on $Y_\et$, the {\it derived modified complex of sheaves} $\widehat{j_Y}M$ on $\overline{Y}_\et$ is defined by the shifted mapping cone
 \[ \begin{split}
 \widehat{j_Y} M := \mathrm{cone}\left( R j_{Y \ast} (M) \overto{\epz_i} i_{Y \ast} i_Y^\ast R j_\ast M  \simeq i_{Y \ast} \bigoplus^{r_1} C^\bullet(M_{\eta})   \right.  \hspace{5em} \\ 
 \left. \overto{\op \mathrm{can}} i_{Y \ast} \bigoplus^{r_1} \widehat{C}^\bullet(M_{\eta}) \right) [-1]
\end{split}\]
\end{definition}
The sheaf cohomology of $\widehat{M}$ are then agrees with the hypercohomology of $\widehat{j_Y}M$ in positive degrees:
\[ \bH^p(\overline{Y}_\et; \widehat{j_Y}M) \cong 
\begin{cases}
 H^p(\overline{Y}_\et;\widehat{M})  & p \ge 0 \\
 \bigoplus_{k=1}^{r_1}H^{p-1}_T(I_{v_k};M_{\eta}) & p < 0
\end{cases}
\]
where $H^\ast_T$ denotes the Tate cohomology.

The construction of the completed \'etale site and the modified sheaf may descend to an arbitrary dense open subscheme of $Y$, especially for $U = \spec \calO_F[\tfrac{1}{2}]$. 
This gives rise to the following commutative diagram of Grothendieck sites:
\begin{equation}\label{diagram}
 \xymatrix{
U_\et \ar[r]^-j \ar[d]_-{j_U}& Y_\et \ar[d]^-{j_Y}& Z_\et  \ar[l]_-i \ar@{=}[d]^-{j_Z} \\
\overline{U}_\et \ar[r]^-{\overline{j}}& \overline{Y}_\et & \overline{Z}_\et \ar[l]_-{\overline{i}} \\
U_\infty \ar[u]^-{i_U} \ar@{=}[r]^-{j_\infty} & Y_\infty \ar[u]_-{i_Y}
}
\end{equation}
where both row and column exhibit recollement in the categories of sheaves.
Although the modification functor $M \mapsto \widehat{M}$ is neither left nor right adjoint, it is compatible with the usual functors involved in recollement as in the following lemma.
Analogous statements also hold for the derived functors.
\begin{lemma}\label{thelemma}
\begin{enumerate}[label=(\alph*)]
 \item The two composite functors $\widehat{j_\ast (-)}, \overline{j}_\ast (\widehat{-}) \colon \Sh(U_\et;\Ab) \to \Sh(\overline{Y}_\et;\Ab)$ are naturally isomorphic.
 \item The two composite functors $\widehat{j^\ast (-)}, \overline{j}^\ast (\widehat{-}) \colon \Sh(Y_\et;\Ab) \to \Sh(\overline{U}_\et;\Ab)$ are naturally isomorphic. 
 \item The two composite functors $\widehat{j_! (-)}, \overline{j}_! (\widehat{-}) \colon \Sh(U_\et;\Ab) \to \Sh(\overline{Y}_\et;\Ab)$ are naturally isomorphic. 
\end{enumerate}
\end{lemma}
\begin{proof}
(a) Fix a sheaf $M$ on $U_\et$.
The sheaf $\overline{j}^\ast \widehat{j_\ast M}$ on $\overline{U_\et}$ is characterized by the triple
\[ \left( j_U^\ast \overline{j}^\ast \widehat{j_\ast M}, (i_U)^\ast \overline{j}^\ast \widehat{j_\ast M} , \phz \right).\]
Over $U_\et$, the sheaf $j_U^\ast \overline{j}^\ast \widehat{j_\ast M}$ is isomorphic to $M$ through a sequence of natural isomorphisms:
\[j_U^\ast \overline{j}^\ast \widehat{j_\ast M} \cong j^\ast j_Y^\ast \widehat{j_\ast M} \cong j^\ast j_\ast M \cong M\]
where the first isomorphism follows from the commutative diagram (\ref{diagram}) and the second from the definition of the modified \'etale sheaf (\ref{modifiedsheaf}).
Likewise, over each real point $v_k$, the stalk is naturally isomorphic to $( M_\eta )_{I_{v_k}}$:
\[i_{v_k}^\ast \overline{j}^\ast \widehat{j_\ast M} \cong (j_\infty)^\ast  i_{v_k}^\ast \widehat{j_\ast M}  \cong ( (j_\ast M)_\eta )_{I_{v_k}} \cong ( M_\eta )_{I_{v_k}}.\]
Because \(\overline{j}^\ast\) simply forgets the finite-prime data, the norm map \(\varphi_k\) at the real point \(v_k\) coincides with that of \(j_\ast M\), noting that \(M_\eta\) and \((j_\ast M)_\eta\) agree. 
Hence, we obtain a natural isomorphism
\[ \overline{j}^\ast \widehat{j_\ast M} \cong \widehat{M},\] 
and the adjoint map 
\[ \widehat{j_\ast M} \to \overline{j}_\ast \widehat{M}\]
is also an isomorphism, since it is an isomorphism on every stalk.

(b) Consider the map adjoint to
\[(\widehat{-}) \to \widehat{j_\ast j^\ast(-) } \simeq \overline{j}_\ast \widehat{j^\ast (-)}.\]
Here, the first map is the unit map $\id \to j_\ast j^\ast$, followed by the modification functor $(\widehat{-}) : \Sh(Y_\et;\Ab) \to \Sh(\overline{Y}_\et;\Ab)$. 
The latter equivalence is from part (a).
Since this map induces an isomorphism on each stalk, and it establishes the desired equivalence.

(c) To show natural isomorphism, consider the map adjoint to
 \[(\widehat{-}) \to \widehat{j^\ast j_!(-)} \simeq \overline{j}^\ast \widehat{ j_!(-)}.\]
In this case, the first map is the unit map $\id \to j^\ast j_!$, followed by the modification functor $(\widehat{-}) \colon \Sh(U_\et;\Ab) \to \Sh(\overline{U}_\et;\Ab)$, and the latter equivalence is from (b).
Again, since this map induces an isomorphism on each stalk, we obtain the desired equivalence.
\end{proof}

Let  $S_\infty$ be the disjoint union of the sites $Y_\infty$ and $Z_\et$.
Then the site $\overline{Y}_\et$ may be regarded as a scheme that adjoins finite prime and real points to $Y_\et$, so that the site morphisms
\[\xymatrix{
U_\et \ar[r]& \overline{Y}_\et & S_\infty \ar[l].
}\]
exhibit a recollement as well.
The Artin-Verdier-Zink duality \cite[3.2.1]{zink} for a sheaf $\bZ/2^n(t)$ on $U_\et$ then states a perfect pairing in this context:
\begin{equation}\label{avzduality}
\begin{split}
\bH^p(\overline{Y}_\et; \widehat{j_Y} j_! \bZ/2^n(t)) \ox H^{3-p}(U_\et; \bZ/2^n(1-t)) \qquad \\
 \to \bH^3(\overline{Y}_\et;\widehat{j_Y} j_!\bG_{m,U}) \cong \bQ/\bZ.
\end{split}
\end{equation}
The pairing is given by the Yoneda pairing using the identification
\begin{align}\label{identification}
\begin{split}
H^{3-p}(U_\et; \bZ/2^n(1-t))&\cong \Ext^{3-p}_{U_\et}(\bZ/2^n(t),\bG_{m,U}) \\
&\cong   \Ext^{3-p}_{Y_\et}(j_! \bZ/2^n(t),j_! \bG_{m,U}) \\ 
&\cong   \Ext^{3-p}_{Y_\et}( j_Y^\ast \widehat{j_Y} j_! \bZ/2^n(t),j_! \bG_{m,U})\\
&\cong   \Ext^{3-p}_{\overline{Y}_\et}(  \widehat{j_Y} j_! \bZ/2^n(t), Rj_{Y \ast} j_! \bG_{m,U})
\end{split}
\end{align}
and the sequence of isomorphisms
\begin{align}\label{seqofiso}
\begin{split}
 \bH^3(\overline{Y}_\et;  Rj_{Y \ast} j_! \bG_{m,U}) & \cong H^3(Y_\et; j_! \bG_{m,U}) \\
 & \cong H^3(\overline{Y}_\et; j_{Y !} j_! \bG_{m,U}) \\
 & \cong  H^3(Y_\et; \widehat{j_! \bG_{m,U}}) \\
 & \cong \bH^3(\overline{Y}_\et; \widehat{j_Y} j_! \bG_{m,U}) \cong  \bQ/\bZ.
 \end{split}
\end{align}
The second isomorphism in (\ref{seqofiso}) follows from the long exact sequence of hypercohomology associated to the short exact sequence (\ref{recolses}) by taking $M= R j_{Y \ast} j_! \bG_{m,U}$ with the fact that the real points do not have higher cohomology as they behave as geometric points in $\overline{Y}_\et$.
Similarly, the third isomorphism follows from the long exact sequence of sheaf cohomology associated to the short exact sequence
\[0 \to j_{Y !}  M \to \widehat{M} \to i_\ast \bigoplus (M_\eta)_{I_{v_k}} \to 0.\]

When $F$ is totally imaginary, two sites $Y_\et$ and $\overline{Y_\et}$ are naturally isomorphic.
Hence, the hypercohomology group $\bH^p(\overline{Y}_\et; \widehat{j_Y} j_! \bZ/2^n(t))$ becomes naturally isomorphic to $H^p(Y_\et;j_! \bZ/2^n(t))$, and (\ref{avzduality}) recovers the original Artin-Verdier duality.

\subsection{Hypersheaf of spectra on the completed \'etale site}
Analyzing the $K(1)$-localized K-theory hypersheaf $\calK$ on $U_\et$ along these functors plays a crucial observation of \cite[2.5]{MR4121155}.
The recollement of sheaves of abelian groups (\ref{abelianrecollement}) extends to the context of the stable $\infty$-category of hypersheaves of spectra:
\[ \xymatrix@C=3em{
\Hyp(U_\et;\Sp) \ar@<1.5ex>[r]^-{j_!} \ar@<-1.5ex>[r]_-{j_\ast} & \Hyp(Y_\et;\Sp)  \ar[l]|-{j^\ast}  \ar@<1.5ex>[r]^-{i^\ast} \ar@<-1.5ex>[r]_-{i^!}  & \Hyp(Z_\et;\Sp)  \ar[l]|-{\ i_\ast}.
}\]
Likewise, the recollement (\ref{zinkrecol}) associated with the completed \'etale site also extends to hypersheaves:
\[ \xymatrix@C=3em{
\Hyp(Y_\et;\Sp) \ar@<1.5ex>[r]^-{j_!} \ar@<-1.5ex>[r]_-{j_\ast} & \Hyp(\overline{Y}_\et;\Sp)  \ar[l]|-{j^\ast}  \ar@<1.5ex>[r]^-{i^\ast} \ar@<-1.5ex>[r]_-{i^!}  & \Hyp(Y_\infty;\Sp)  \ar[l]|-{\ i_\ast}.
}\]
As we did with the sheaves of abelian groups (\ref{fixedpointidentification}), the hypersheaf $i^\ast j_\ast \calE$ may be naturally identified with a collection of homotopy fixed points under the $I_{v_k}$-action on each stalk $\calE_{\eta}$:
\[i^\ast j_\ast \calE \simeq \bigvee_{k=1}^{r_1} (\calE_{\eta})^{hI_{v_k}}.\]
In light of the \autoref{derivedmodifiedsheaf} of the derived modified complex of sheaves, we give the definition of the modified hypersheaf of spectra.

\begin{definition}
Given a hypersheaf of spectra $\calE$ on the \'etale site $U_\et$, its associated {\it modified hypersheaf} $\widehat{j}\calE$ on $\overline{U}_\et$ is defined as a homotopy fiber:
\begin{equation}\label{eq2}
\widehat{j} \calE :=\Fib\left( j_\ast \calE \longrightarrow i_\ast i^\ast j_\ast \calE \simeq i_\ast \bigvee_{k=1}^{r_1} (\calE_{\eta})^{hI_{v_k}} \overto{\mathrm{can}} i_\ast \bigvee_{k=1}^{r_1} (\calE_{\eta})^{t I_{v_k}} \right) 
\end{equation}
where $(\calE_{\eta})^{tI_{v_k}}$ is the Tate spectrum, a homotopy cofiber of the hypernorm  $N_h: (\calE_{\eta})_{hI_{v_k}} \to (\calE_{\eta})^{hI_{v_k}}$ associated with $I_{v_k}$-action on the spectrum $\calE_{\eta}$; see \cite[Ch 3]{Jardine}.
Equivalently, this can be defined by the homotopy pullback:
\[\xymatrix{
\widehat{j}\calE \ar[r] \ar[d] \ar@{}[dr]|(0.2)\lrcorner& i_\ast  \bigvee_{k=1}^{r_1} (\calE_{\eta})_{hI_{v_k}} \ar[d]^-{N_h} \\
j_\ast \calE \ar[r] & i_\ast i^\ast j_\ast \calE.
}\]
\end{definition}

\begin{note}\label{note}
When $\calE$ is a hypersheaf of ring spectra, both $j_\ast \calE$ and $(\calE_{\eta})^{tI_{v_k}}$ inherit ring structures, and the maps in (\ref{eq2}) are ring maps.
Therefore, the modified hypersheaf $\widehat{j}\calE$ inherits a $j_\ast \calE$-module structure.
In general, if $\calF$ is an $\calE$-module, then $\widehat{j}\calF$ inherits a $j_\ast \calE$-module structure.
\end{note}

In the case of the $K(1)$-local algebraic K-theory sheaf $\calK$, its stalk $\calK_{\eta}$ is the 2-completed periodic complex K-theory spectrum $KU_2^\w$ whose Tate construction is known to vanish \cite[6.1]{MR3349325}.
Hence, we get a 2-equivalence:
\begin{equation}\label{eq1}
\widehat{j}\calK \simeq j_\ast \calK.
\end{equation}
In \cite[2.5]{MR4121155}, Blumberg-Mandell found a hypercohomological description of $\Fib(\ka)$ as a global section $j_! \calK_{U_\et}(Y_\et)$, up to $p$-completion.
Therefore, in the context of the completed \'etale site, we get $j_{Y \ast} j_! \calK_{U_\et}(\overline{Y}_\et)$ as a model for $\Fib(\ka)$.
Here we state another model for the purpose of our pairing.
\begin{proposition}\label{fibashypshf}
The homotopy fiber $\Fib(\ka)$ of $K(1)$-local K-theoretic completion map 
\begin{equation}\label{hypcohdescrip}
 \ka\colon L_{K(1)}K(\calO_F[\tfrac{1}{2}]) \to \prod_{\nu\in \calS} L_{K(1)}K(F_\nu)
\end{equation}
 is equivalent to the global section $\widehat{j_Y} j_! \calK_{U_\et} (\overline{Y}_\et)$ of the hypersheaf of spectra $\widehat{j_Y} j_! \calK_{U_\et}$ after 2-completion.
\end{proposition}
\begin{proof}
This can be proven by the sequence of equivalences
\[\widehat{j_Y} j_! \calK_{U_\et}\simeq \overline{j}_! \widehat{j_U} \calK_{U_\et} \simeq \overline{j}_! j_{U \ast} \calK_{U_\et} \simeq j_{Y \ast} j_! \calK_{U_\et}\]
where the first equivalence follows from  the derived version of the \autoref{thelemma}(a) and the second equivalence follows from (\ref{eq1}).
The last morphism is adjoint to the natural map
\[ j_{U \ast} \calK_{U_\et} \to j_{U \ast} j^\ast j_! \calK_{U_\et} \simeq \overline{j}^\ast j_{Y \ast} j_! \calK_{U_\et}\]
which is an equivalence since it is an equivalence in each stalk.
\end{proof}

Therefore, the spectrum $\Fib(\ka)$ has two 2-equivalent hypercohomological models as a global section of hypersheaves
\begin{equation}\label{fibka}
 \widehat{j_Y} j_! \calK_{U_\et}  \overto\simeq j_{Y \ast} j_!  \calK_{U_\et}
\end{equation}
and both admit natural $j_{Y \ast} j_\ast \calK_{U_\et}$-module structure, as discussed in \autoref{note}.
Under this identification, the pairing
\[\Fib(\ka) \w L_{K(1)}K(\calO_F[\tfrac{1}{2}]) \to \Fib(\ka)\]
coincides with the pairing
\begin{equation}\label{pairing}
 \overline{j}_\ast j_{U \ast} \calK \w \overline{j}_! \widehat{j_U} \calK \simeq \overline{j}_\ast j_{U \ast} \calK \w \overline{j}_! j_{U \ast} \calK \to \overline{j}_! j_{U \ast} \calK
\end{equation}
given by $ \overline{j}_\ast j_{U \ast} \calK$-module structure on $\overline{j}_! \widehat{j_U} \calK$ followed by the global section; see \cite[2.7]{MR4121155}.

\section{Construction of spectral sequences and pairing}\label{s3}
In this section, we construct spectral sequences for the spectra $L_{K(1)}K(\calO_F[\tfrac{1}{2}])^{/2^n}$ and $\Fib(\ka)^{/2^n}$.
Note that the site morphism $j\colon Y_\et \to \overline{Y}_\et$ in (\ref{modifiedrecollement}) transfers categorical properties that allow the construction of the Thomason-Jardine descent spectral sequence \cite[3.1-3]{MR1803955}.
After that, we provide a proof of the main result, \autoref{mainthm}.

\subsection{Construction of spectral sequences and pairing}
For the global section $\calK^{/2^n}_{U_\et} (U_\et)$, we use the same spectral sequence as in \cite[\S4]{MR4121155} arising from the Whitehead tower:
\[ \cdots \to W^{t+1} \calK^{/2^n}_{U_\et} (U_\et)  \to W^t \calK^{/2^n}_{U_\et} (U_\et) \to \cdots\]
where $W^\bullet$ is the truncation functor $\ta^{\bullet \ge}$.
We write $C^t \calK^{/2^n}_{U_\et}$ for the homotopy cofiber
\[C^t \calK^{/2^n}_{U_\et} := \mathrm{Cof}(W^{t+1} \calK^{/2^n}_{U_\et}  \to W^{t} \calK^{/2^n}_{U_\et} ).\]
The filtration then gives the descent spectral sequence with the $E_2$-page
\begin{equation}\label{ssK}
 E_2^{s,t} =\pi_{-s+t}  C^t \calK^{/2^n}_{U_\et} =H^{s-t}(U_\et; \bZ/2^n(\tfrac{t}{2}))
\end{equation}
which abuts to the colimit
\[ \colim \pi_{-s+t}  W^\bullet \calK^{/2^n}_{U_\et}  \cong \pi_{t-s}\calK_{U_\et}^{/2^n}(U_\et).\]
This spectral sequence collapses at the $E_4$-page with finitely many nonzero terms in each total degree \cite[2.10]{MR1803955}. Therefore this strongly converges.

There are two spectral sequences for $\Fib(\ka)^{/2^n}$, one for each model $j_{Y\ast} j_! \calK_{U_\et}^{/2^n}$ and $\widehat{j_Y} j_! \calK_{U_\et}^{/2^n}$; see (\ref{fibka}).
We first consider the Whitehead tower followed by the functor $j_!$:
\[ \cdots \to j_! W^{t+1} \calK^{/2^n}_{U_\et} ({Y}_\et)  \to  j_! W^t \calK^{/2^n}_{U_\et} ({Y}_\et) \to \cdots.\]
We abbreviate each level
\[ \calX^t := j_! W^t \calK^{/2^n}_{U_\et}  \simeq \Fib\left( j_\ast W^t \calK^{/2^n}_{U_\et} \to  i_\ast i^\ast  j_\ast  W^t \calK^{/2^n}_{U_\et}\right) \]
and write the homotopy cofiber as
\[C^t \calX := \mathrm{Cof}( \calX^{t+1} \to \calX^t).\]
Then the spectral sequence associated to this filtration has $E_2$-page:
\[  E^{s,t}_2 \left(\overline{Y}_\et; j_{Y \ast} j_!\calK^{/2^n}_{U_\et} \right)  = \pi_{-s+t} C^t \calX({Y}_\et) =  H^s (Y_\et; j_! \bZ/2^n(\tfrac{t}{2}))  \]
that approaches the colimit
\[ \colim \pi_{-s+t} \calX^\bullet ({Y}_\et) \cong \pi_{-s+t} \left( j_! \calK^{/2^n}_{U_\et} ({Y}_\et)\right).\]
The spectral sequence also collapses at the $E_4$-page as in (\ref{ssK}) and converges strongly as the $\holim \calX^\bullet$ is trivial and $H^s (Y_\et; j_! \bZ/2^n(\tfrac{t}{2}))$ is only non-zero in a finite range.

For the second model, consider a tower
\[ \cdots \to \widehat{j_Y} j_! W^{t+1} \calK^{/2^n}_{U_\et} (\overline{Y}_\et)  \to \widehat{j_Y} j_! W^t \calK^{/2^n}_{U_\et} (\overline{Y}_\et) \to \cdots.\]
We abbreviate
\[ \widehat{\calX}^t := \widehat{j_Y} j_! W^t \calK^{/2^n}_{U_\et}  \simeq \Fib\left( \widehat{j_Y} j_\ast W^t \calK^{/2^n}_{U_\et} \to  \widehat{j_Y} i_\ast i^\ast  j_\ast  W^t \calK^{/2^n}_{U_\et}\right) \]
and write
\[C^t \widehat{\calX} := \mathrm{Cof}( \widehat{\calX}^{t+1} \to \widehat{\calX}^t)\]
for the homotopy cofiber.
Then the $E_2$-page of the associated spectral sequence is given by the following
\begin{equation}\label{ss1}
 E^{s,t}_2\left(\overline{Y}_\et; \widehat{j_Y}j_!\calK^{/2^n}_{U_\et} \right) =  \pi_{-s+t} C^t \widehat{\calX} (\overline{Y}_\et) \cong \bH^s (\overline{Y}_\et;\widehat{j_Y} j_! \bZ/2^n(\tfrac{t}{2})).
\end{equation}
This spectral sequence approaches to the colimit
\[ \colim \pi_{-s+t} \widehat{\calX}^\bullet (\overline{Y}_\et) \cong \pi_{-s+t} \left( \widehat{j_Y} j_! \calK^{/2^n}_{U_\et} (\overline{Y}_\et)\right)\]
and conditionally converges as $\holim \widehat{\calX}^\bullet \simeq \ast$.
\begin{proposition}
The spectral sequence (\ref{ss1}) collapses at the $E_4$-page and strongly converges.
\end{proposition}
\begin{proof}
Note that the hypercohomology group $\bH^s (\overline{Y}_\et;\widehat{j_Y} j_! \bZ/2^n(\tfrac{t}{2}))$ vanishes for $s \ge 4$ as the completed \'etale site $\overline{Y}_\et$ has cohomological dimension 3 for constructible sheaves; see \cite[2.6]{zink}.
This makes the spectral sequence (\ref{ss1}) into shifted half plane spectral sequence and so it suffices to show the left half plane ($s<0$) collapses at the $E_4$-page with finitely many non-zero terms in each total degree.
The left half plane consists of
\[ E^{s,t}_2 = \bigoplus_{k=1}^{r_1} H^{s-1}_T(I_{v_k}; \bZ/2^n(\tfrac{t}{2})) = \bigoplus_{k=1}^{r_1}  H_{-s}(I_{v_k}; \bZ/2^n(\tfrac{t}{2}))\]
where the latter term is group homology which are essentially governed by the real embeddings.
The site morphism $i_Y\colon Y_\infty \to \overline{Y}_\et$ induces a map on hypercohomology spectra: 
 \[\bH \left( \overline{Y}_\et;\overline{j}_! \widehat{j_U}\calK^{/2^n}_{U_\et} \right) \to \bH \left(Y_\infty;(i_Y)^\ast \overline{j}_! \widehat{j_U}\calK^{/2^n}_{U_\et} \right) \simeq \prod_{k=1}^{r_1} \bH \left( \spec \bR;(i_{v_k})^\ast \overline{j}_! \widehat{j_U}\calK^{/2^n}_{U_\et} \right)\]
 where each product factor in the latter can be identified with the real periodic K-theory spectrum mod $2^n$:
 \[(i_{v_k})^\ast \overline{j}_! \widehat{j_U}\calK^{/2^n}_{U_\et} \simeq (j_\infty)_! (i_{v_k})^\ast \widehat{j_U}\calK^{/2^n}_{U_\et} \simeq (KU^{/2^n})_{hC_2} \simeq KO^{/2^n}.\]
The induced map of the descent spectral sequence on the $E_2$-page gives:
\[\bH^s \left(\overline{Y}_\et;\widehat{j_Y} j_! \bZ/2^n(\tfrac{t}{2}) \right) \to \bigoplus_{k=1}^{r_1} \bH^s \left( \bR; (i_{v_k})^\ast \widehat{j_Y} j_! \bZ/2^n(\tfrac{t}{2}) \right)\]
which is an isomorphism for $s<0$.
The latter spectral sequence is the homotopy orbit spectral sequence associated with $KU_{hC_2}$ under the complex conjugate action. 
This collapses after $E_4$-page with finitely many non-zero terms in each total degree.
Since this forces the desired spectral sequence (\ref{ss1}) to converges strongly as well, this finishes the proof.
\end{proof}

The pairing property of the Whitehead tower functor $W^\bullet$ will then induce a pairing of spectral sequences (see \cite[\S4]{MR4121155}) and we claim this is the desired pairing.
\begin{theorem}\label{thmss}
The descent spectral sequences
\begin{align*}
& E_2^{s,t}(U_\et;\calK^{/2^n}_{U_\et} ) = H^s(U_\et; {\bZ/2^n(\tfrac{t}{2})}) \Rightarrow \pi_{t-s} \calK^{/2^n}_{U_\et} (U_\et) \\
&E_2^{s,t}\left(\overline{Y}_\et; \widehat{j_Y}j_!\calK^{/2^n}_{U_\et} \right) = \bH^s(\overline{Y}_\et; \widehat{j_Y}j_! \bZ/2^n(\tfrac{t}{2})) \Rightarrow \pi_{t-s} \widehat{j_Y}j_!\calK^{/2^n}_{U_\et} (\overline{Y}_\et) \\
& E_2^{s,t}\left(\overline{Y}_\et; j_{Y \ast} j_!\calK^{/2^n}_{U_\et} \right) = H^s(Y_\et;  j_! \bZ/2^n(\tfrac{t}{2})) \Rightarrow \pi_{t-s} j_{Y \ast} j_!\calK^{/2^n}_{U_\et} (\overline{Y}_\et)
\end{align*}
 admit a pairing of spectral sequences
 \[ E_r^{s,t}(U;\calK^{/2^n}_{U_\et} ) \ox E_r^{s',t'}\left(Y_\et; \widehat{j_Y}j_!\calK^{/2^n}_{U_\et}\right)  \to E_r^{s+s',t+t'}\left(Y_\et; j_{Y \ast} j_!\calK^{/2^n}_{U_\et} \right).\]
This pairing converges to the following pairing
\[\pi_{t-s} {\calK^{/2^n}_{U_\et}} (U_\et)  \ox \pi_{t'-s'} \widehat{j_Y}j_!\calK^{/2^n}_{U_\et} (Y_\et) \to \pi_{t+t'-(s+s')} j_{Y \ast} j_!\calK^{/2^n}_{U_\et}  (Y_\et) \]
induced from the weak equivalences
\[{\calK^{/2^n}_{U_\et}} (U_\et)  \simeq \calK^{/2^n}(\calO_F[\tfrac{1}{2}]) \quad \text{and}\quad \widehat{j_Y}j_!\calK^{/2^n}_{U_\et}  (\overline{Y}_\et) \simeq j_{Y \ast} j_!\calK^{/2^n}_{U_\et}  (\overline{Y}_\et) \simeq \Fib(\ka)^{/2^n}\]
with the pairing
\[ \calK^{/2^n}(\calO_F[\tfrac{1}{2}])\w \Fib(\ka)^{/2^n}\ \to  \Fib(\ka)^{/2^n}.\]
\end{theorem}
\begin{proof}
Given a pairing (\ref{pairing}), the pairing property of the Whitehead tower functor $W^\bullet$ induces a pairing 
\begin{align*}
W^t \calK^{/2^n}_{U_\et}& (U_\et) \w \widehat{\calX}^{t'}(\overline{Y}_\et) = j_{Y \ast} j_\ast W^t \calK^{/2^n}_{U_\et} (\overline{Y}_\et) \w \widehat{j_Y} j_! W^{t'} \calK^{/2^n}_{U_\et} (\overline{Y}_\et)  \\
& \to  j_{Y \ast} j_\ast W^t \calK^{/2^n}_{U_\et} (\overline{Y}_\et) \w j_{Y \ast} j_! W^{t'} \calK^{/2^n}_{U_\et} (\overline{Y}_\et)  \\
&\to j_{Y \ast} j_! W^{t+t'}(\calK^{/2^n}_{U_\et} (\overline{Y}_\et) \w \calK^{/2^n}_{U_\et} (\overline{Y}_\et)) \\
& \to  j_{Y \ast} j_! W^{t+t'} \calK^{/2^n}_{U_\et} (\overline{Y}_\et)  = \calX^{t+t'}(\overline{Y}_\et),
\end{align*}
and this induces a pairing of spectral sequences.
A note after \autoref{fibashypshf} then identifies this with our model (\ref{pairing}).
\end{proof}

Next, we show that the pairing on the $E_2$-page coincides with the Yoneda pairing used in (\ref{identification}).
\begin{theorem}\label{yonedapairing}
Under the canonical isomorphisms
\begin{align*}
H^{\ast}(U_\et; \bZ/2^n(\tfrac{t}{2}))&\cong \Ext^{\ast}_U(\bZ/2^n(t'),\bZ/2^n(\tfrac{t}{2}+\tfrac{t'}{2})) \\
&\cong   \Ext^{\ast}_Y(j_! \bZ/2^n(t'), j_! \bZ/2^n(\tfrac{t}{2}+\tfrac{t'}{2})) \\ 
&\cong   \Ext^{\ast}_Y( j_Y^\ast \widehat{j_Y} j_! \bZ/2^n(t'), j_! \bZ/2^n(\tfrac{t}{2}+\tfrac{t'}{2}))\\
&\cong   \Ext^{\ast}_{\overline{Y}}(  \widehat{j_Y} j_! \bZ/2^n(t'),Rj_{Y \ast} j_! \bZ/2^n(\tfrac{t}{2}+\tfrac{t'}{2})),
\end{align*}
the multiplication on the $E_2$-term in \autoref{thmss} coincides with the Yoneda pairing
\[\begin{split}
\Ext^{\ast}_{\overline{Y}}\left( \widehat{j_Y} j_! \bZ/2^n(\tfrac{t'}{2}),Rj_{Y \ast} j_! \bZ/2^n(\tfrac{t}{2}+\tfrac{t'}{2}) \right) \ox \bH^\ast(\overline{Y}_\et; \widehat{j_Y}j_! \bZ/2^n(\tfrac{t'}{2})) \qquad \\
 \to  \bH^\ast(\overline{Y}_\et; Rj_{Y \ast} j_! \bZ/2^n(\tfrac{t'}{2})) \cong  H^\ast(\overline{Y}_\et; j_! \bZ/2^n(\tfrac{t'}{2})) .
\end{split}\]
\end{theorem}
\begin{proof}
The multiplication on the $E_2$-term is induced by the map of pairs
\[j_{Y \ast} j_\ast C^t \calK^{/2^n}_{U_\et}(\overline{Y}_\et) \w C^{t'} \widehat{\calX}(\overline{Y}_\et) \to C^{t+t'} \calX(\overline{Y}_\et). \]
The homotopy cofibers $j_{Y \ast} j_\ast C^\bullet \calK^{/2^n}_{U_\et}$, $j_{Y \ast} C^\bullet \calX$ and $C^\bullet \widehat{\calX}$ are models for the \'etale hypersheaf $\SI^\bullet j_{Y \ast} j_\ast H\bZ/2^n(\tfrac{t}{2})$, $\SI^\bullet j_{Y \ast} j_! H\bZ/2^n(\tfrac{t'}{2})$ and $\SI^\bullet \widehat{j_Y} j_! H\bZ/2^n(\tfrac{t'}{2})$, respectively.
The induced map on homotopy groups of global sections is then the composite of the cup product
\begin{align*}
 H^\ast(U_\et;& \bZ/2^n(\tfrac{t}{2}) ) \ox \bH^\ast(\overline{Y}_\et; \widehat{j_Y} j_! \bZ/2^n (\tfrac{t'}{2}) ) \\
& \cong  \bH^\ast(\overline{Y}_\et;  j_{Y \ast} j_\ast H\bZ/2^n(\tfrac{t}{2}))  \ox \bH^\ast(\overline{Y}_\et; \widehat{j_Y} j_! H\bZ/2^n (\tfrac{t'}{2})) \\
& \to  \pi_{-\ast} \left( \bH(\overline{Y}_\et;  j_{Y \ast} j_\ast H\bZ/2^n(\tfrac{t}{2}))  \w \bH(\overline{Y}_\et; \widehat{j_Y} j_! H\bZ/2^n (\tfrac{t'}{2})) \right) \\
& \to  \pi_{-\ast}  \bH(\overline{Y}_\et; j_{Y \ast} j_\ast H\bZ/2^n(\tfrac{t}{2}) \w \widehat{j_Y} j_! H\bZ/2^n (\tfrac{t'}{2}) ),
\end{align*}
and the map of \'etale hypersheaves of spectra on $\overline{Y}_\et$
\begin{align}\label{hypshpair}
\begin{split}
 j_{Y \ast} & j_\ast H\bZ/2^n(\tfrac{t}{2}) \w \widehat{j_Y} j_! H\bZ/2^n (\tfrac{t'}{2}) \\
& \to j_{Y \ast} j_\ast H\bZ/2^n(\tfrac{t}{2}) \w j_{Y \ast} j_! H\bZ/2^n (\tfrac{t'}{2}) \\
&\qquad \to j_{Y \ast} j_! H\bZ/2^n (\tfrac{t+t'}{2}),
\end{split}
\end{align}
where the former map is induced by the natural transformation $\widehat{j_Y} \to Rj_{Y \ast}$ and the latter map is induced by the pairing.
Using the stable Dold-Kan correspondence, the map (\ref{hypshpair}) can be written as follows
\begin{equation}\label{eq4}
 H (R j_{Y \ast} j_\ast \bZ/2^n(\tfrac{t}{2})) \w H \widehat{j_Y} j_! \bZ/2^n (\tfrac{t'}{2}) \to H (Rj_{Y \ast} j_! \bZ/2^n (\tfrac{t+t'}{2})).
\end{equation}
We claim that this pairing factors through the derived tensor product on the derived category of sheaves of abelian groups on $\overline{Y}_\et$.
Then the rest of the proof follows from the relationship between the cup product and the Yoneda product in the derived category of sheaves of abelian groups on $Y_\et$, cf. \cite[\S5.1]{MR0559531}

Over each real point $v_k$, the map (\ref{eq4}) reduces to the composite
\begin{equation}\label{eq5}
 \begin{split}
 (H\bZ/2^n(\tfrac{t}{2}))^{h C_2} \w (H\bZ/2^n(\tfrac{t'}{2}))_{h C_2} \qquad \qquad \\
 \to  (H\bZ/2^n(\tfrac{t}{2}))^{h C_2} \w (H\bZ/2^n(\tfrac{t'}{2}))^{h C_2}  \\
 \to (H\bZ/2^n(\tfrac{t+t'}{2}))^{h C_2} , 
\end{split}
\end{equation}
where the former map is induced by the hypernorm on $H\bZ/2^n(\tfrac{t'}{2})$ and the latter map is an $\ef$-multiplication on $(H\bZ/2^n(\tfrac{t+t'}{2}))^{h C_2}$ and these are induced by the pairing of discrete rings.
Over the site $Y_\et$, the map (\ref{eq4}) reduces to the pairing
\begin{equation}\label{eq6}
 H (Rj_\ast \bZ/2^n(\tfrac{t}{2})) \w H j_! \bZ/2^n (\tfrac{t'}{2}) \to H j_! \bZ/2^n (\tfrac{t+t'}{2}),
\end{equation}
and this factors through the derived tensor product as in \cite[\S4]{MR4121155}.
Now we need to show two pairing maps (\ref{eq5}) and (\ref{eq6}) are compatible along the gluing morphisms.
Over an each real point $v_k$, this boils down to show the following diagram to commutes:
\[\xymatrix{
 (H\bZ/2^n(\tfrac{t}{2}))^{h C_2} \w (H\bZ/2^n(\tfrac{t'}{2}))_{h C_2}  \ar[r] \ar[d]^-{\id \w Nm} &  (H\bZ/2^n(\tfrac{t+t'}{2}))_{h C_2} \ar[d]^-{Nm} \\
( H \bZ/2^n(\tfrac{t}{2}) )^{hC_2} \w ( H \bZ/2^n (\tfrac{t'}{2}))^{hC_2} \ar[r] & (H \bZ/2^n (\tfrac{t+t'}{2}))^{hC_2}.
}\]
This diagram commute since the hypernorm map is $(H \bZ/2^n(\tfrac{t}{2}) )^{hC_2}$ equivariant.

Hence the multiplication on the $E_2$-term can be written as follows:
\begin{align*}
 H^\ast(U_\et;& \bZ/2^n(\tfrac{t}{2}) ) \ox \bH^\ast(\overline{Y}_\et; \widehat{j_Y} j_! \bZ/2^n (\tfrac{t'}{2}) ) \\
& \cong  \bH^\ast(\overline{Y}_\et;  j_{Y \ast} j_\ast H\bZ/2^n(\tfrac{t}{2}))  \ox \bH^\ast(\overline{Y}_\et; \widehat{j_Y} j_! H\bZ/2^n (\tfrac{t'}{2})) \\
& \to  H^{\ast} \left( \bH_{\Ab}(\overline{Y}_\et;  R j_{Y \ast} j_\ast \bZ/2^n(\tfrac{t}{2}))  \ox^\bL \bH_{\Ab}(\overline{Y}_\et; \widehat{j_Y} j_! \bZ/2^n (\tfrac{t'}{2})) \right) \\
& \to  \bH^\ast (\overline{Y}_\et; R j_{Y \ast} j_\ast \bZ/2^n(\tfrac{t}{2}) \ox^\bL \widehat{j_Y} j_! \bZ/2^n (\tfrac{t'}{2}) ) \\
& \to  \bH^\ast (\overline{Y}_\et; R j_{Y \ast} j_\ast \bZ/2^n(\tfrac{t}{2}) \ox^\bL R j_{Y \ast} j_! \bZ/2^n (\tfrac{t'}{2}) ) \\
& \to  \bH^\ast (\overline{Y}_\et; R j_{Y \ast} j_! \bZ/2^n (\tfrac{t+t'}{2}) ) = H^\ast( Y_\et;j_! \bZ/2^n (\tfrac{t+t'}{2}) ).
\end{align*}
\end{proof}

\subsection{Proof of the main theorem}
Before proving the main theorem, we analyze the multiplicative structure of the mod $2^n$ Moore spectrum $M_{\bZ/2^n}$.
Unlike odd prime, the mod 2 Moore spectrum $M_{\bZ/2}$ does not admit a unital multiplication. 
For $n \ge 3$, however, $M_{\bZ/{2^n}}$ admits an $\bE_1$-algebra structure \cite{burklund}.
Furthermore, this multiplicative structure is not unique \cite{MR760188}.

\begin{lemma}\label{moore}
For $n\ge 3$, any multiplication on the Moore spectrum $M_{\bZ/{2^n}}$ makes the following composite an equivalence 
\[M_{\bZ/{2^n}} \to F(M_{\bZ/{2^n}}, M_{\bZ/{2^n}}) \to F(M_{\bZ/{2^n}}, M_{\bZ/{2^\infty}}),\]
where the first map is adjoint to a multiplication map and the second map is induced by the canonical map $M_{\bZ/{2^n}}\to M_{\bZ/{2^\infty}}$.
Under this equivalence, the map
\[M_{\bZ/{2^n}} \w M_{\bZ/{2^n}} \simeq F(M_{\bZ/{2^n}}, M_{\bZ/2^\infty}) \w M_{\bZ/{2^n}} \to M_{\bZ/2^\infty}\]
induced by the evaluation is the same as the composite of the multiplication on $M_{\bZ/{2^n}}$
and the inclusion $M_{\bZ/{2^n}} \to M_{\bZ/2^\infty}$.
\end{lemma}
\begin{proof}
Consider the following map of cofiber sequences:
 \[\xymatrix{
\bS \ar[r]^-{2^n} \ar[d] & \bS \ar[r] \ar[d]^-{1/2^n} & M_{\bZ/{2^n}} \ar[d] \ar[r] & \bS^1 \ar[d]^-\id \\
\bS \ar[r]  & \bS[\tfrac{1}{2}] \ar[r] & M_{\bZ/{2^\infty}} \ar[r] & \bS^1.
}\]
Applying the functor $F(M_{\bZ/{2^n}},-)$, we get the following
\[\xymatrix{
F(M_{\bZ/{2^n}}, \bS) \ar[r] \ar[d]^-{1/{2^n}} & F(M_{\bZ/{2^n}}, M_{\bZ/{2^n}}) \ar[d] \ar[r] & F(M_{\bZ/{2^n}}, \bS^1) \ar[d]^-\id \ar@/^1.5pc/@{-->}[l] \\
F(M_{\bZ/{2^n}}, \bS[\tfrac{1}{2}]) \ar[r] & F(M_{\bZ/{2^n}}, M_{\bZ/{2^\infty}}) \ar[r]^-\simeq & F(M_{\bZ/{2^n}}, \bS^1)
}\]
where the upper row admits a dashed splitting.
\end{proof}

\begin{proof}[Proof of the K-theoretic Tate-Poitou duality at prime 2]
Consider a composite
 \[ \Fib(\ka) \w L_{K(1)}K(\calO_F[\tfrac{1}{2}]) \to \Fib(\ka) \to \SI^{-1}I_{\bZ_2}\bS\]
 of a $L_{K(1)}K(\calO_F[\tfrac{1}{2}])$-action and the canonical map in (\ref{dualityclass}).
\begin{equation}\label{lasteqn}
\begin{split}
  \Fib(\ka) \to \SI^{-1}I_{\bZ_2}  (L_{K(1)}K(\calO_F[\tfrac{1}{2}])) \qquad\qquad\qquad\\
  \simeq \SI^{-1}I_{\bQ/\bZ}  (L_{K(1)}K(\calO_F[\tfrac{1}{2}] \w M_{\bZ/2^\infty} )).
\end{split}
\end{equation}
Since both sides are 2-complete, it suffices to show (\ref{lasteqn}) becomes an equivalence after smashing with mod $2^n$ Moore spectrum $M_{2^n}$ for $n\ge 3$.
Using the canonical weak equivalence
\[\SI^{-1}I_{\bQ/\bZ}(L_{K(1)}K(\calO_F[\tfrac{1}{2}] \w M_{\bZ/2^\infty} )) /2^n  \simeq \SI^{-1}I_{\bQ/\bZ}(L_{K(1)}K(\calO_F[\tfrac{1}{2}]  /2^n  )),\]
the induced map 
 \[ 
  \Fib(\ka) / 2^n \to \SI^{-1}I_{\bQ/\bZ}(L_{K(1)}K(\calO_F[\tfrac{1}{2}]  /2^n  ))
\]
is adjoint to the map
\begin{equation}\label{lastmap}
\Fib(\ka) / 2^n \w L_{K(1)}K(\calO_F[\tfrac{1}{2}])  /2^n  \to \SI^{-1}I_{\bQ/\bZ} \bS.
\end{equation}
By \autoref{moore}, the map (\ref{lastmap}) is the composite of the multiplication
\[\Fib(\ka)  / 2^n  \w L_{K(1)}K(\calO_F[\tfrac{1}{2}]) / 2^n  \to \Fib(\ka) / 2^n \]
and the map
\[ \Fib(\ka)  / 2^n \to  \Fib(\ka) \w M_{\bZ/2^\infty}  \to \SI^{-1}I_{\bZ_2}\bS,\]
where the last map is (\ref{dualityclassadj}), adjoint to $u_{\calO_F}$.

Because  (\ref{avzduality}) is a perfect pairing, \autoref{thmss} and \autoref{yonedapairing} give a pairing of the spectral sequences which is perfect pairing on $E_2$-page.
This perfect pairing descends to $E_\infty$-page, and this implies that (\ref{lastmap}) induces a perfect pairing on homotopy groups
\[ \pi_q (\Fib(\ka) / 2^n ) \ox \pi_{-q-1} (L_{K(1)}K(\calO_F[\tfrac{1}{2}])  /2^n)  \to \bQ/\bZ.
\]
Thus we deduce that (\ref{lasteqn}) is a weak equivalence and this finishes the proof.
\end{proof}


\bibliography{KT_bib}
\bibliographystyle{alpha}

\end{document}